\documentclass[envcountsame]{llncs}
\usepackage{amssymb}
\usepackage{amsmath}

\newcommand{\set}[2]{\left\{ #1 \mid #2 \right\}}
\newcommand{\mfrac}[2]{#1 / #2}

\newcommand{\mand}{\text{ and }}
\newcommand{\eand}{\qquad \text{and} \qquad}
\newcommand{\Ni}{\mathbb N_1}

\newcommand{\pref}[1]{\mathrm{pref}_{#1}}
\newcommand{\suff}[1]{\mathrm{suff}_{#1}}
\newcommand{\rfact}[2]{\mathrm{rfact}_{#2}^{#1}}

\newcommand{\kae}[1]{\sim_{#1}}
\newcommand{\factors}[1]{\mathcal{F}_{#1}}
\newcommand{\fc}[2]{\mathcal{P}_{#2}^{(#1)}}
\newcommand{\fcsturm}[1]{q^{(#1)}}
\newcommand{\fcu}[2]{\mathcal{U}_{#2}^{(#1)}}
\newcommand{\fcl}[2]{\mathcal{L}_{#2}^{(#1)}}

\begin{document}

\title{
Variations of the Morse-Hedlund Theorem for $k$-Abelian Equivalence
}

\author{
Juhani Karhum\" aki
\inst{1}\thanks{
Partially supported by the Academy of Finland under grants 251371 and 257857.
}
\and
Aleksi Saarela
\inst{1}$\mbox{}^\star$
\and
Luca Q. Zamboni
\inst{1}\inst{2}\thanks{
Partially supported by a FiDiPro grant (137991) from the Academy of Finland and by ANR grant {\sl SUBTILE}.
}
}

\institute{
Department of Mathematics and Statistics \& FUNDIM,
University of Turku,
FI-20014 Turku, Finland
\and
Universit\'e de Lyon,
Universit\'e Lyon 1, CNRS UMR 5208,
Institut Camille Jordan,
43 boulevard du 11 novembre 1918,
F69622 Villeurbanne Cedex, France
\and
\email{karhumak@utu.fi},
\email{amsaar@utu.fi},
\email{lupastis@gmail.com}
}

\maketitle

\begin{abstract}
In this paper we investigate local to global phenomena for a new family of complexity functions of infinite words indexed by $k \in \Ni \cup \{+\infty\}$ where $\Ni$ denotes the set of positive integers. Two finite words $u$ and $v$ in $A^*$ are said to be $k$-Abelian equivalent if for all $x \in A^*$ of length less than or equal to $k$, the number of occurrences of $x$ in $u$ is equal to the number of occurrences of $x$ in $v$. This defines a family of equivalence relations $\sim_k$ on $A^*$, bridging the gap between the usual notion of Abelian equivalence (when $k = 1$) and equality (when $k = +\infty$). Given an infinite word $w \in A^\omega$, we consider the associated complexity function $\mathcal P^{(k)}_w : \Ni \rightarrow \Ni$ which counts the number of $k$-Abelian equivalence classes of factors of $w$ of length $n$. As a whole, these complexity functions have a number of common features: Each gives a characterization of periodicity in the context of bi-infinite words, and each can be used to characterize Sturmian words in the framework of aperiodic one-sided infinite words. Nevertheless, they also exhibit a number of striking differences, the study of which is one of the main topics of our paper.
\end{abstract}

\section{Introduction}

A fundamental problem in both mathematics
and  computer science is to describe local constraints which imply global regularities.
A splendid example of this phenomena may be found in the framework of combinatorics on words.
In their seminal papers \cite{MoHe38,MoHe40}, G. A. Hedlund and M. Morse proved that a
bi-infinite word $w$ is periodic if and only if for some positive integer $n,$ $w$ contains at most $n$ distinct factors of length $n.$
In other words, it describes the exact borderline between periodicity and aperiodicity of words in terms of the {\it factor complexity function} which counts the number of distinct factors of each length $n.$ An analogous result was established some thirty years later by E. Coven and G.A. Hedlund in the framework of Abelian equivalence. They show that a bi-infinite word is periodic if and only if for some positive integer $n$ all factors of $w$ are Abelian equivalent. Thus once again it is possible to distinguish  between periodic and aperiodic words on a local level by counting the number of Abelian equivalence classes of factors of length $n.$

In this paper we study the local to global behavior for a  new family of complexity functions $\fc{k}{w}$ of infinite words indexed by $k \in \Ni \cup \{+\infty\}$ where $\Ni=\{1,2,3,\ldots \}$ denotes the set of positive integers. Let $k \in \Ni  \cup \{+\infty\}$ and $A$ be a finite non-empty set. Two finite words $u$ and $v$ in $A^*$ are said to be $k$-{\it Abelian equivalent} if for all $x\in A^*$ of length less than or equal to $k,$ the number of occurrences of $x$ in $u$ is equal to the number of occurrences of $x$ in  $v.$ This defines a family of  equivalence relations $\thicksim_k$ on $A^*,$ bridging the gap between the usual notion of Abelian equivalence (when $k=1$) and equality (when $k=+\infty).$  Abelian equivalence of words has long been a subject of great interest (see for instance Erd\"os problem, \cite{CaRiSaZa11,CoHe73,CuRa11abelian,De79,Ke92,PuZa13,RiSaZa10,RiSaZa11,Sa09jalc}).
Although the notion of $k$-Abelian equivalence is quite new, there are already a number of recent papers on the topic
\cite{HuKaSaSa11maurer,HuKa11,HuKaSa12ehrenfeucht,KaPuSa12dlt,Mesa13dlt,KaSaZa_jcta}.

Given an infinite word $w \in A^\omega,$
we consider the associated complexity function $\fc{k}{w} :\Ni \rightarrow \Ni$ which counts the number of  $k$-Abelian equivalence classes of factors of $w$ of length $n.$ Thus $\fc{\infty}{w}$ corresponds to the usual factor complexity while $\fc{1}{w}$ corresponds to Abelian complexity.
As it turns out, each  intermediary complexity function  $\fc{k}{w}$ for $2\leq k <+\infty$ can be used to detect periodicity of words. In order to describe this connection, we will make use of the following
auxiliary function  first
discovered in \cite{KaSaZa_jcta}:
\begin{equation*}
  \fcsturm{k}(n) = \begin{cases}
      n + 1 & \text{if $n \leq 2k - 1$} \\
      2k & \text{if $n \geq 2k$}
  \end{cases}.
\end{equation*}

\noindent As a starting point of our research, we list two classical results on factor and Abelian complexity in connection with periodicity, and their $k$-Abelian counterparts proved by the authors in
\cite{KaSaZa_jcta}. We note that in each case, the first two items are included in the third.

\begin{theorem}\label{MH}
Let $w$ be a bi-infinite word over a finite alphabet. Then the following properties hold:
\begin{itemize}
\item (M. Morse, G.A. Hedlund, \cite{MoHe38})
The word $w$ is periodic if and only if  $\fc{\infty}{w}(n)\leq n$ for some $n\geq 1$.
\item (E.M. Coven, G.A. Hedlund, \cite{CoHe73})
The word $w$ is periodic  if and only if  $\fc{1}{w}(n)=1$ for some $n\geq 1$.
\item  The word $w$ is periodic if and only if $\fc{k}{w}(n)<\fcsturm{k}(n)$ for some $k\in \Ni\cup \{+\infty\}$ and $n\geq 1.$
\end{itemize}
\end{theorem}

\noindent
Also, each complexity provides a characterization for an important class of binary words, the so-called \textit{Sturmian} words:
%
% that is, the infinite words
%$\omega\in A^\nats$ with  $\p_{\omega}(n)=n+1$ for all $n\geq 0$.

\begin{theorem}\label{SturmianChar}
Let $w$ be an aperiodic one-sided infinite word. Then the following properties hold:
\begin{itemize}
\item (M. Morse, G.A. Hedlund, \cite{MoHe40}). The word $w$ is Sturmian if and only if  $\fc{\infty}{w}(n) = n+1$ for all $n \geq 1$.
\item (E.M. Coven, G.A. Hedlund, \cite{CoHe73}). The word $w$ is Sturmian if and only if $\fc{1}{w}(n)=2$ for all $n\geq 1$.
\item  The word $w$ is Sturmian if and only if $\fc{k}{w}(n)=\fcsturm{k}(n)$ for all $k\in \Ni\cup \{+\infty\}$ and $n\geq 1.$
\end{itemize}
\end{theorem}

However, in other respects,  these various complexities  exhibit radically different behaviors. For instance, in the context of one-sided infinite words, the first item  in Theorem~\ref{MH} gives rise to a characterization of ultimately periodic words, while for the other two, the result holds in only one direction: If  $\fc{k}{w}(n)<\fcsturm{k}(n)$ for some $k\in \Ni$ and $n\geq 1$ then $w$ is ultimately periodic, but not conversely (see \cite{KaSaZa_jcta}).
For instance in the simplest case when $k=1,$ it is easy to see that if $w$ is the ultimately periodic word $01^\omega,$ then for each positive integer $n$ there are precisely two Abelian classes of factors of $w$ of length $n.$  However, the same is true of the (aperiodic) {\it Fibonacci infinite word}
\[w=010010100100101001\ldots\]
defined as the fixed point of the morphism $0\mapsto 01,$ $1\mapsto 0.$
Analogously, in Theorem~\ref{SturmianChar} the first item holds true without the added assumption that $w$ be aperiodic, while the other two items do not. Another striking difference between them is in their rate of growth.
Consider for instance the binary Champernowne word
\[ {\mathcal C} =01101110010111011110001001\ldots\]
obtained by concatenating the binary representation
of the consecutive natural numbers. Let $w$ denote the morphic image of $\mathcal C$
under the Thue-Morse morphism $\tau$ defined by $0\mapsto 01$ and $1\mapsto 10$.  Then while $\fc{\infty}{w}(n)$ has exponential growth, it can be shown that $\fc{1}{w}(n)\leq 3$ for all $n.$
Yet another fundamental disparity concerns the difference $\fc{k}{w}(n+1)- \fc{k}{w}(n).$
For the factor complexity, one always has $\fc{\infty}{w}(n+1)- \fc{\infty}{w}(n)\geq 0,$ while  for general $k$ this inequality is far from being true.

A primary objective in this paper is to study the asymptotic lower and upper complexities
defined by
\begin{equation*}
  \fcl{k}{w}(n) = \min_{m \geq n} \fc{k}{w}(m)
  \eand
  \fcu{k}{w}(n) = \max_{m \leq n} \fc{k}{w}(m).
\end{equation*}

Surprisingly these quantities might deviate from one another quite drastically.
Indeed,
one of our main results is to compute these values for the famous Thue-Morse word.
We show that the upper limit is logarithmic,
while the lower limit is just constant, in fact at most $8$ in the case $k=2.$ This is quite unexpected considering the Thue-Morse word is both pure morphic and Abelian periodic (of period $2).$
If we however allow more general words, then
we obtain much stronger evidence of the non-existence of gaps
in low $k$-Abelian complexity classes.
We construct  uniformly recurrent infinite words
having arbitrarily low upper limit and just constant lower limit.
The concept of $k$-Abelian complexity
also leads to many interesting open questions.
We conclude the paper in Sect. \ref{sec:conclusion}
by mentioning some of these problems.

\section{Preliminaries}

Let $\Sigma$ be a finite non-empty set called the \emph{alphabet}.
The set of all finite words over $\Sigma$ is denoted by $\Sigma^*$ and
the set of all (right) infinite words is denoted by $\Sigma^\omega$.
The set of positive integers is denoted by $\Ni$.

Let $w \in \Sigma^\omega$.
The word $w$ is \emph{periodic}
if there is $u \in \Sigma^*$ such that $w = u^\omega$,
and \emph{ultimately periodic}
if there are $u, v \in \Sigma^*$ such that $w = v u^\omega$.
If $w$ is not ultimately periodic, then it is \emph{aperiodic}.
Let $u = a_0 \dots a_{m - 1}$ and $a_0, \dots, a_{m - 1} \in \Sigma$.
The \emph{prefix} of length $n$ of $u$ is
$\pref{n}(u) = a_0 \dots a_{n - 1}$ and
the \emph{suffix} of length $n$ of $u$ is
$\suff{n}(u) = a_{m - n} \dots a_{m - 1}$.
If $0 \leq i \leq m$, then the notation
$\rfact{i}{n}(u) = a_i \dots a_{i + n - 1}$
is used.
The length of a word $u$ is denoted by $|u|$ and
the number of occurrences of another word $x$ as a factor of $u$ by $|u|_x$.
Two words $u, v \in \Sigma^*$ are \emph{Abelian equivalent}
if $|u|_a = |v|_a$ for all $a \in \Sigma$.

Let $k \in \Ni$.
Two words $u, v \in \Sigma^*$ are \emph{$k$-Abelian equivalent}
if $|u|_x = |v|_x$ for all words $x$ of length at most $k$.
$k$-Abelian equivalence is denoted by $\kae{k}$.
If the length of $u$ and $v$ is at least $k - 1$,
then $u \kae{k} v$ if and only if
$|u|_x = |v|_x$ for all words $x$ of length $k$
and $\pref{k - 1}(u) = \pref{k - 1}(v)$
and $\suff{k - 1}(u) = \suff{k - 1}(v)$.

Let $w \in \Sigma^\omega$.
The set of factors of $w$ of length $n$ is denoted by $\factors{w}(n)$.
The \emph{factor complexity} of $w$ is the function
\begin{math}
   \fc{\infty}{w}: \Ni \to \Ni
\end{math}
defined by
\begin{equation*}
   \fc{\infty}{w}(n) = \# \factors{w}(n).
\end{equation*}
Let $k \in \Ni$.
The \emph{$k$-Abelian complexity} of $w$ is the function
\begin{math}
   \fc{k}{w}: \Ni \to \Ni
\end{math}
defined by
\begin{equation*}
   \fc{k}{w}(n) = \# (\factors{w}(n) / \kae{k}).
\end{equation*}

Factor complexity functions are always increasing,
and even strictly increasing for aperiodic words.
For $k$-Abelian complexity this is not true.
This is why we define
\emph{upper $k$-Abelian complexity} $\fcu{k}{w}$ and
\emph{lower $k$-Abelian complexity} $\fcl{k}{w}$ by
\begin{equation*}
   \fcu{k}{w}(n) = \max_{m \leq n} \fc{k}{w}(m)
   \eand
   \fcl{k}{w}(n) = \min_{m \geq n} \fc{k}{w}(m).
\end{equation*}
These two functions can be significantly different.
For example, if $w$ is the Thue-Morse word and $k \geq 2$,
then $\fcu{k}{w}(n) = \Theta(\log n)$ and $\fcl{k}{w}(n) = \Theta(1)$.
This will be proved in Sect. \ref{sec:tm}.

The Abelian complexity of a binary word $w \in \{0, 1\}^\omega$
can be determined using the formula
\begin{equation} \label{eq:abelcomplbinary}
   \fc{1}{w}(n)
   = \max \set{|u|_1}{u \in \factors{n}(w)} - \min \set{|u|_1}{u \in \factors{n}(w)} + 1.
\end{equation}

For $k \in \Ni \cup \{\infty\}$, let
\begin{math}
   \fcsturm{k}: \Ni \to \Ni
\end{math}
be the function defined by
\begin{equation*}
   \fcsturm{k}(n) = \begin{cases}
       n + 1 & \text{if $n \leq 2k - 1$} \\
       2k & \text{if $n \geq 2k$}
   \end{cases}.
\end{equation*}
The significance of this function is that if $w$ is Sturmian,
then $\fc{k}{w} = \fcsturm{k}$. This is further discussed in Sect. \ref{sec:minimal}.

There are large classes of words for which
the $k$-Abelian complexities are of the same order for many values of $k$.
This is shown in the next two lemmas.
Thus when analyzing the growth rate of the $k$-Abelian complexity of a word,
it may be sufficient to analyze the Abelian or 2-Abelian complexity.

\begin{lemma} \label{lem:fckfc1}
Let $w \in \{0,1\}^\omega$ be such that
every factor of $w$ of length $k$ contains at most one occurrence of 1.
Then
\begin{math}
   \fc{k}{w}(n) = \Theta(\fc{1}{w}(n)).
\end{math}
\end{lemma}

\begin{proof}
Two factors of $w$ are $k$-Abelian equivalent if and only if
they are Abelian equivalent and
have the same prefixes and suffixes of length $k - 1$.
\qed\end{proof}

\begin{lemma} \label{lem:fixedpointuniform}
Let $k, m \geq 2$ and let $w$ be a fixed point of an $m$-uniform morphism $h$.
Let $i$ be such that $m^i \geq k - 1$.
Then
\begin{math}
   \fc{k}{w}(m^i (n + 1)) = O(\fc{2}{w}(n)).
\end{math}
\end{lemma}

\begin{proof}
Every factor of $w$ of length $m^i (n + 1)$ can be written as $p h^i(u) q$,
where $u$ is a factor of $w$ of length $n$ and $|pq| = m^i$.
The $k$-Abelian equivalence class of $p h^i(u) q$
is determined by $p$, $q$ and the 2-Abelian equivalence class of $u$.
\qed\end{proof}

In particular,
Lemma \ref{lem:fixedpointuniform} can be applied to the Thue-Morse word
to analyze its $k$-Abelian complexity
once the behavior of its 2-Abelian complexity is known.

It has been shown that there are many words for which
the $k$-Abelian and $(k + 1)$-Abelian complexities are similar,
but there are also many words for which they are very different.
For example,
there are words having bounded $k$-Abelian complexity but
linear $(k + 1)$-Abelian complexity.
These words can even be assumed to be $k$-Abelian periodic.
This is shown in the next lemma.

\begin{lemma} \label{lem:perlin}
For every $k \geq 1$
there is a $k$-Abelian periodic word $w$ such that
\begin{math}
   \fc{k + 1}{w}(n) = \Theta(n).
\end{math}
\end{lemma}

\begin{proof}
Let $w \in \{0, 1\}^\omega$ and let $h$ be the morphism defined by
\begin{equation*}
   0 \mapsto 0^{k + 1} 1 0^{k - 1} 1,
   \qquad
   1 \mapsto 0^{k} 1 0^{k} 1.
\end{equation*}
Then $h(w)$ is $k$-Abelian periodic and
\begin{math}
   \fc{k + 1}{h(w)}((2k + 2) n) = \Theta(\fc{1}{w}(n)).
\end{math}
The claim follows because there are words $w$ with linear Abelian complexity.
\qed\end{proof}

\section{Minimal $k$-Abelian Complexities} \label{sec:minimal}

In this section classes of words with small $k$-Abelian complexity are studied.
Some well-known results about factor complexity are compared to
results on $k$-Abelian complexity proved in \cite{KaSaZa_jcta}.
It should be expected that ultimately periodic words have low complexity,
and this is indeed true for $k$-Abelian complexity,
although the $k$-Abelian complexity of some ultimately periodic words
is higher that the $k$-Abelian complexity of some aperiodic words.
For many complexity measures,
Sturmian words have the lowest complexity among aperiodic words.
This is also true for $k$-Abelian complexity.

We recall the famous theorem of Morse and Hedlund \cite{MoHe38}
characterizing ultimately periodic words in terms of factor complexity.
This theorem can be generalized for $k$-Abelian complexity:
If
\begin{math}
   \fc{k}{w}(n) < \fcsturm{k}(n)
\end{math}
for some $n$, then $w$ is ultimately periodic,
and if $w$ is ultimately periodic, then
\begin{math}
   \fc{\infty}{w}(n)
\end{math}
is bounded.
This was proved in \cite{KaSaZa_jcta}.

If $k$ is finite, then this generalization
does not give a characterization of ultimately periodic words,
because the function $\fcsturm{k}$ is bounded.
In fact, it is impossible to characterize ultimately periodic words
in terms of $k$-Abelian complexity.
For example, the word $0^{2k - 1} 1^\omega$
has the same $k$-Abelian complexity as every Sturmian word.
On the other hand, for every ultimately periodic word $w$
there is a finite $k$ such that
\begin{math}
   \fc{k}{w}(n) < \fcsturm{k}(n)
\end{math}
for all sufficiently large $n$.

The theorem of Morse and Hedlund has a couple of immediate consequences.
The words $w$ with $\fc{\infty}{w}(n) = n + 1$ for all $n$ are,
by definition, Sturmian words.
Thus the following classification is obtained:
\begin{itemize}
\item $w$ is ultimately periodic
$\Leftrightarrow$
$\fc{\infty}{w}$ is bounded.

\item $w$ is Sturmian
$\Leftrightarrow$
$\fc{\infty}{w}(n) = n + 1$ for all $n$.

\item $w$ is aperiodic and not Sturmian
$\Leftrightarrow$
$\fc{\infty}{w}(n) \geq n + 1$ for all $n$ and
$\fc{\infty}{w}(n) > n + 1$ for some $n$.
\end{itemize}
This can be generalized for $k$-Abelian complexity
if the equivalences are replaced with implications:
\begin{itemize}
\item $w$ is ultimately periodic
$\Rightarrow$
$\fc{k}{w}$ is bounded.

\item $w$ is Sturmian
$\Rightarrow$
$\fc{k}{w} = \fcsturm{k}$.

\item $w$ is aperiodic and not Sturmian
$\Rightarrow$
$\fc{k}{w}(n) \geq \fcsturm{k}(n)$ for all $n$ and
$\fc{k}{w}(n) > \fcsturm{k}(n)$ for some $n$.
\end{itemize}
For $k = 1$ this follows from the theorem of Coven and Hedlund \cite{CoHe73}.
For $k \geq 2$ it follows from a theorem in \cite{KaSaZa_jcta}.

The above result means that
one similarity between factor complexity and $k$-Abelian complexity is that
Sturmian words have the lowest complexity among aperiodic words.
Another similarity is that ultimately periodic words have bounded complexity,
but the largest values can be arbitrarily high:
For every $n$, there is a finite word $u$
having every possible factor of length $n$.
Then $\fc{k}{u^\omega}(n)$ is as high as it can be for any word,
i.e. the number of $k$-Abelian equivalence classes of words of length $n$.

Another direct consequence of the theorem of Morse and Hedlund is that
there is a gap between constant complexity and
the complexity of Sturmian words.
For $k$-Abelian complexity there cannot be a gap
between bounded complexities and $\fcsturm{k}$,
because the function $\fcsturm{k}$ itself is bounded.
However, the question whether there is a gap above bounded complexity
is more difficult.
The answer is that there is no such gap,
even if only uniformly recurrent words are considered.
This is proved in Sect. \ref{sec:arbslow}.

\section{$k$-Abelian Complexity of the Thue-Morse Word} \label{sec:tm}

In this section the $k$-Abelian complexity of the Thue-Morse word is analyzed.
Before that, the Abelian complexity of a closely related word is determined.

Let $\sigma$ be the morphism defined by
\begin{math}
   \sigma(0) = 01, \sigma(1) = 00.
\end{math}
Let
\begin{equation*}
   S = 0100010101000100 \dots
\end{equation*}
be the \emph{period-doubling word},
which is the fixed point of $\sigma$, see \cite{Da00}.

The Abelian complexity of $S$ is completely determined
by the recurrence relations in the following lemma
and by the first two values
\begin{math}
   \fc{1}{S}(1) = \fc{1}{S}(2) = 2.
\end{math}

\begin{lemma} \label{lem:srec}
For $n \geq 1$,
\begin{align*}
   \fc{1}{S}(4n - 1) &= \fc{1}{S}(n) + 1, &\quad
   \fc{1}{S}(4n    ) &= \fc{1}{S}(n)    , \\
   \fc{1}{S}(4n + 1) &= \fc{1}{S}(n) + 1, &\quad
   \fc{1}{S}(4n + 2) &= \fc{1}{S}(n) + 1.
\end{align*}
\end{lemma}

\begin{proof}
%The structure of $S$ allows to write recursion formulas for
%$\factors{4n - 1}(S)$, $\factors{4n}(S)$,
%$\factors{4n + 1}(S)$ and $\factors{4n + 2}(S)$
%from which the result can be concluded.
%The full proof can be found in the arXiv prepint of this paper.
Let
\begin{math}
  p_n = \min \set{|u|_1}{u \in \factors{n}(S)}
  \mand
  q_n = \max \set{|u|_1}{u \in \factors{n}(S)}.
\end{math}
For $a \in \{0, 1\}$, $\sigma^2(a) = 010a$.
Because
\begin{align*}
  \factors{4n - 1}(S) =
      &\set{\sigma^2(a_1 \dots a_{n - 1}) 010}{a_1 \dots a_{n - 1} \in \factors{n - 1}(S)} \cup\\
      &\set{10 a_1 \sigma^2(a_2 \dots a_n)}{a_1 \dots a_n \in \factors{n}(S)} \cup\\
      &\set{0 a_1 \sigma^2(a_2 \dots a_n) 0}{a_1 \dots a_n \in \factors{n}(S)} \cup\\
      &\set{a_1 \sigma^2(a_2 \dots a_n) 01}{a_1 \dots a_n \in \factors{n}(S)},
\end{align*}
\begin{align*}
  \factors{4n}(S) =
      &\set{\sigma^2(a_1 \dots a_n)}{a_1 \dots a_n \in \factors{n}(S)} \cup\\
      &\set{10 a_1 \sigma^2(a_2 \dots a_n) 0}{a_1 \dots a_n \in \factors{n}(S)} \cup\\
      &\set{0 a_1 \sigma^2(a_2 \dots a_n) 01}{a_1 \dots a_n \in \factors{n}(S)} \cup\\
      &\set{a_1 \sigma^2(a_2 \dots a_n) 010}{a_1 \dots a_n \in \factors{n}(S)},
\end{align*}
\begin{align*}
  \factors{4n + 1}(S) =
      &\set{\sigma^2(a_1 \dots a_n) 0}{a_1 \dots a_n \in \factors{n}(S)} \cup\\
      &\set{10 a_1 \sigma^2(a_2 \dots a_n) 01}{a_1 \dots a_n \in \factors{n}(S)} \cup\\
      &\set{0 a_1 \sigma^2(a_2 \dots a_n) 010}{a_1 \dots a_n \in \factors{n}(S)} \cup\\
      &\set{a_1 \sigma^2(a_2 \dots a_{n + 1})}{a_1 \dots a_{n + 1} \in \factors{n + 1}(S)},
\end{align*}
\begin{align*}
  \factors{4n + 2}(S) =
      &\set{\sigma^2(a_1 \dots a_n) 01}{a_1 \dots a_n \in \factors{n}(S)} \cup\\
      &\set{10 a_1 \sigma^2(a_2 \dots a_n) 010}{a_1 \dots a_n \in \factors{n}(S)} \cup\\
      &\set{0 a_1 \sigma^2(a_2 \dots a_{n + 1})}{a_1 \dots a_{n + 1} \in \factors{n + 1}(S)} \cup\\
      &\set{a_1 \sigma^2(a_2 \dots a_{n + 1}) 0}{a_1 \dots a_{n + 1} \in \factors{n + 1}(S)},
\end{align*}
it can be seen that
\begin{align*}
  p_{4n - 1} &= p_n + n - 1, & q_{4n - 1} &= q_n + n    ,\\
  p_{4n    } &= p_n + n    , & q_{4n    } &= q_n + n    ,\\
  p_{4n + 1} &= p_n + n    , & q_{4n + 1} &= q_n + n + 1,\\
  p_{4n + 2} &= p_n + n    , & q_{4n + 2} &= q_n + n + 1.
\end{align*}
The claim follows because
\begin{math}
  \fc{1}{S}(n) = q_n - p_n + 1
\end{math}
for all $n$.
\qed\end{proof}

\begin{lemma} \label{lem:fcs}
For $n \geq 1$ and $m \geq 0$,
\begin{align*}
   \fc{1}{S}(n) = O(\log n), \quad
   \fc{1}{S}((2 \cdot 4^m + 1) / 3) = m + 2, \quad
   \fc{1}{S}(2^m) = 2.
\end{align*}
\end{lemma}

\begin{proof}
Follows from Lemma \ref{lem:srec} by induction.
\qed\end{proof}

The Abelian complexity of $S$
has a logarithmic upper bound and a constant lower bound.
These bounds are the best possible increasing bounds.

\begin{theorem} \label{thm:fcs}
\begin{math}
   \fcu{1}{S}(m) = \Theta(\log n)
   \mand
   \fcl{1}{S}(m) = 2.
\end{math}
\end{theorem}

\begin{proof}
Follows from Lemma \ref{lem:fcs}.
\qed\end{proof}

Now, let $\tau$ be the Thue-Morse morphism defined by
\begin{math}
   \tau(0) = 01, \tau(1) = 10.
\end{math}
Let
\begin{equation*}
   T = 0110100110010110 \dots
\end{equation*}
be the Thue-Morse word,
which is a fixed point of $\tau$.
The first few values of $\fc{2}{T}$ are
\begin{equation*}
   2,4,6,8,6,8,10,8,6,8,8,10,10,10,8,8,6,8,10,10.
\end{equation*}

The 2-Abelian equivalence of factors of $T$
can be determined with the help of the following lemma.

\begin{lemma} \label{lem:2abeleq}
Words $u, v \in \{0, 1\}^*$ are 2-Abelian equivalent if and only if
\begin{equation*}
   |u| = |v|, \qquad
   |u|_{00} = |v|_{00}, \qquad
   |u|_{11} = |v|_{11} \eand
   \pref{1}(u) = \pref{1}(v).
\end{equation*}
\end{lemma}

\begin{proof}
The ``only if'' direction follows immediately
from the definition of 2-Abelian equivalence.
For the other direction,
it follows from the assumptions that
\begin{math}
   |u|_{01} + |u|_{10} = |v|_{01} + |v|_{10}.
\end{math}
In any word $w \in \{0, 1\}^*$,
$|w|_{01}$ and $|w|_{10}$ can differ by at most one.
If
\begin{math}
   |w|_{01} + |w|_{10}
\end{math}
is even, then
\begin{math}
   |w|_{01} = |w|_{10}.
\end{math}
If it is odd and $\pref{1}(w) = 0$, then
\begin{math}
   |w|_{01} = |w|_{10} + 1.
\end{math}
If it is odd and $\pref{1}(w) = 1$, then
\begin{math}
   |w|_{01} + 1 = |w|_{10}.
\end{math}
This means that
\begin{math}
   |u|_{01} = |v|_{01} \mand |u|_{10}= |v|_{10}
\end{math}
and $u$ and $v$ are 2-Abelian equivalent.
\qed\end{proof}

The following lemma states that
if $u$ is a factor of $T$,
then the numbers $|u|_{00}$ and $|u|_{11}$
can differ by at most one.

\begin{lemma} \label{lem:00a11}
In the image of any word under $\tau$,
between any two occurrences of $00$ there is an occurrence of $11$
and vice versa.
\end{lemma}

\begin{proof}
00 can only occur in the middle of $\tau(10)$ and
11 can only occur in the middle of $\tau(01)$.
The claim follows because 10's and 01's alternate in all binary words.
\qed\end{proof}

Let $u$ be a factor of $T$.
If $|u|$ and $|u|_{00} +  |u|_{11}$ are given,
then there are at most 4 possibilities
for the 2-Abelian equivalence class of $u$.
This is stated in a different way in the next lemma.
First we define a function $\phi$ as follows.
If $w = a_1 \dots a_n$,
then $\phi(w) = b_1 \dots b_{n - 1}$, where
$b_i = 0$ if $a_i a_{i + 1} \in \{01, 10\}$ and
$b_i = 1$ if $a_i a_{i + 1} \in \{00, 11\}$.
If $w = a_1 a_2 \dots$ is an infinite word,
then $\phi(w) = b_1 b_2 \dots$ is defined in an analogous way.

\begin{lemma} \label{lem:phi}
Let $u_1, \dots, u_n$ be factors of $T$.
If $\phi(u_1), \dots, \phi(u_n)$ are Abelian equivalent,
then $u_1, \dots, u_n$
are in at most 4 different 2-Abelian equivalence classes.
\end{lemma}

\begin{proof}
The numbers
\begin{math}
   |u_i|_{00} + |u_i|_{11} = |\phi(u_i)|_1
\end{math}
are equal for all $i$;
let this number be $m$.
By Lemma \ref{lem:00a11},
\begin{math}
  \{|u_i|_{00}, |u_i|_{11}\} = \{\lfloor m / 2 \rfloor, \lceil m / 2 \rceil\}.
\end{math}
There are at most four different possible values for the triples
\begin{math}
   (|u_i|_{00}, |u_i|_{11}, \pref{1}(u_i)).
\end{math}
The claim follows from Lemma \ref{lem:2abeleq}.
\qed\end{proof}

Now it can be proved that the 2-Abelian complexity of $T$
is of the same order as the Abelian complexity of $\phi(T)$.
It is known that $\phi(T)$ is actually the period-doubling word $S$.

\begin{lemma} \label{lem:fctphit}
For $n \geq 2$,
\begin{math}
   \fc{1}{S}(n - 1) \leq \fc{2}{T}(n) \leq 4 \fc{1}{S}(n - 1).
\end{math}
\end{lemma}

\begin{proof}
If the factors of $T$ of length $n$ are
$u_1, \dots, u_m$,
then the factors of $\phi(T)$ of length $n - 1$ are
$\phi(u_1), \dots, \phi(u_m)$.
If $u_i$ and $u_j$ are 2-Abelian equivalent,
then $\phi(u_i)$ and $\phi(u_j)$ are Abelian equivalent,
so the first inequality follows.
The second inequality follows from Lemma \ref{lem:phi}
\qed\end{proof}

\begin{lemma} \label{lem:fct}
For $n \geq 1$ and $m \geq 0$,
\begin{align*}
   \fc{2}{T}(n) = O(\log n), \quad
   \fc{2}{T}((2 \cdot 4^m + 4) / 3) = \Theta(m), \quad
   \fc{2}{T}(2^m + 1) \leq 8.
\end{align*}
\end{lemma}

\begin{proof}
Follows from Lemmas \ref{lem:fctphit} and \ref{thm:fcs}.
\qed\end{proof}

The $k$-Abelian complexity of $T$
behaves in a similar way as
the Abelian complexity of $S$.

\begin{theorem}
Let $k \geq 2$.
Then
\begin{equation*}
   \fcu{k}{T}(m) = \Theta(\log n)
   \eand
   \fcl{k}{T}(m) = \Theta(1).
\end{equation*}
\end{theorem}

\begin{proof}
Follows from Lemmas \ref{lem:fct} and \ref{lem:fixedpointuniform}.
\qed\end{proof}

\section{Arbitrarily Slowly Growing $k$-Abelian Complexities} \label{sec:arbslow}

In this section it is studied
whether there is a gap above bounded $k$-Abelian complexity.
This question can be formalized in two ways:
\begin{itemize}
\item Does there exists an increasing unbounded function $f: \Ni \to \Ni$
such that for every infinite word $w$
either $\fc{k}{w}$ is bounded or $\fc{k}{w} = \Omega(f)$?

\item Does there exists an increasing unbounded function $f: \Ni \to \Ni$
such that for every infinite word $w$
either $\fc{k}{w}$ is bounded or $\fc{k}{w} \ne O(f)$?
\end{itemize}
The first question has already been answered negatively in Sect. \ref{sec:tm}.
The answer to the second question is also negative,
even if only uniformly recurrent words are considered.
A construction proving this is given below.

Let $n_1, n_2, \dots$ be a sequence of integers greater than 1.
Let $m_j = n_1 \dots n_j$ for $j = 0, 1, 2, \dots$.
Let $a_i = 0$ if the greatest $j$ such that $m_j | i$ is even and
$a_i = 1$ otherwise.
Let $U = a_1 a_2 a_3 \dots$.

\begin{lemma} \label{lem:uunrec}
The word $U$ is uniformly recurrent.
\end{lemma}

\begin{proof}
For every factor $u$ of $U$,
there is a $j$ such that $u$ is a factor of $\pref{m_j - 1}(U)$.
Because $U \in \{\pref{m_j - 1}(U) 0, \pref{m_j - 1}(U) 1\}^\omega$,
every factor of $U$ of length $m_j + |u| - 1$ contains $u$.
\qed\end{proof}

\begin{lemma} \label{lem:fcu}
For every $n$, let $n'$ be such that $m_{n' - 1} < n \leq m_{n'}$.
Then
\begin{equation*}
   \fc{1}{U}(n) \leq n' + 1.
\end{equation*}
For all $J \geq 1$, if $n = 2 \sum_{j = 1}^J (m_{2 j} - m_{2 j - 1})$, then
\begin{equation*}
   \fc{1}{U}(n) \geq \frac{n' + 1}{2}.
\end{equation*}
For all $j \geq 1$,
\begin{equation*}
   \fc{1}{U}(m_j) = 2.
\end{equation*}
\end{lemma}

\begin{proof}
Formula \eqref{eq:abelcomplbinary} will be used repeatedly in this proof.
Another important simple fact is that
if $a, b, c$ are integers and $c$ divides $a$, then
\begin{math}
   \left\lfloor \mfrac{(a + b)}{c} \right\rfloor
   = \mfrac{a}{c} + \left\lfloor \mfrac{b}{c} \right\rfloor .
\end{math}

For all $n \geq 1$,
\begin{equation*}
   |\pref{n}(U)|_1
   = \sum_{i = 1}^\infty (-1)^{i + 1} \left\lfloor \frac{n}{m_i} \right\rfloor,
\end{equation*}
and for all $n \geq 1$ and $l \geq 0$,
\begin{equation*}
   |\rfact{l}{n}(U)|_1
   = |\pref{n + l}(U)|_1 - |\pref{l}(U)|_1
   = \sum_{i = 1}^\infty (-1)^{i + 1} \left(
       \left\lfloor \frac{n + l}{m_i} \right\rfloor -
       \left\lfloor \frac{l}{m_i} \right\rfloor \right).
\end{equation*}
For all $i$,
\begin{equation*}
   \left\lfloor \frac{n + l}{m_i} \right\rfloor -
       \left\lfloor \frac{l}{m_i} \right\rfloor
   \in \left\{ \left\lfloor \frac{n}{m_i} \right\rfloor,
       \left\lceil \frac{n}{m_i} \right\rceil \right\}.
\end{equation*}
Moreover, for every $l$ there is an $i'$ such that
\begin{equation*}
   \left\lfloor \frac{n + l}{m_i} \right\rfloor -
       \left\lfloor \frac{l}{m_i} \right\rfloor
   = \begin{cases}
       1 & \text{if $n' \leq i \leq i'$} \\
       0 & \text{if $i \geq i'$}
   \end{cases},
\end{equation*}
so
\begin{equation*}
   \sum_{i = n'}^\infty (-1)^{i + 1} \left(
       \left\lfloor \frac{n + l}{m_i} \right\rfloor -
       \left\lfloor \frac{l}{m_i} \right\rfloor \right)
   \in \left\{ 0, (-1)^{n' + 1} \right\}.
\end{equation*}
Thus there are at most $n' + 1$ possible values for $|\rfact{l}{n}(U)|_1$
and $\fc{1}{U}(n) \leq n' + 1$.

Consider the second claim.
Let
\begin{math}
   n = 2 \sum_{j = 1}^J (m_{2 j} - m_{2 j - 1})
\end{math}
and
\begin{math}
   l = m_{2J + 1} - n / 2 .
\end{math}
Then
\begin{equation*}
   |\rfact{l}{n}(U)|_1 - |\pref{n}(U)|_1
   = \sum_{i = 1}^\infty (-1)^{i + 1} \left(
       \left\lfloor \frac{n + l}{m_i} \right\rfloor -
       \left\lfloor \frac{l}{m_i} \right\rfloor -
       \left\lfloor \frac{n}{m_i} \right\rfloor \right)
\end{equation*}
and for $i \leq 2J$
\begin{align*}
   &\left\lfloor \frac{n + l}{m_i} \right\rfloor -
       \left\lfloor \frac{l}{m_i} \right\rfloor -
       \left\lfloor \frac{n}{m_i} \right\rfloor \\
   = &\frac{m_{2J + 1} + \sum_{(i + 1) / 2 \leq j \leq J} (m_{2 j} - m_{2 j - 1})}{m_i}
       + \left\lfloor \frac{\sum_{1 \leq j < (i + 1) / 2} (m_{2 j} - m_{2 j - 1})}{m_i} \right\rfloor \\
   &- \frac{m_{2J + 1} - \sum_{(i + 1) / 2 \leq j \leq J} (m_{2 j} - m_{2 j - 1})}{m_i}
       - \left\lfloor - \frac{\sum_{1 \leq j < (i + 1) / 2} (m_{2 j} - m_{2 j - 1})}{m_i} \right\rfloor \\
   &- \frac{2 \sum_{(i + 1) / 2 \leq j \leq J} (m_{2 j} - m_{2 j - 1})}{m_i}
       - \left\lfloor \frac{2 \sum_{1 \leq j < (i + 1) / 2} (m_{2 j} - m_{2 j - 1})}{m_i} \right\rfloor \\
   = &\left\lfloor \frac{s}{m_i} \right\rfloor
       - \left\lfloor - \frac{s}{m_i} \right\rfloor
       - \left\lfloor \frac{2 s}{m_i} \right\rfloor,
\end{align*}
where $s = \sum_{1 \leq j < (i + 1) / 2} (m_{2 j} - m_{2 j - 1})$.
If $i$ is even, then $m_i / 2 \leq s < m_i$,
and if $i$ is odd and $i > 1$, then $m_{i - 1} / 2 \leq s < m_{i - 1}$.
Thus
\begin{equation*}
   \left\lfloor \frac{s}{m_i} \right\rfloor
       - \left\lfloor - \frac{s}{m_i} \right\rfloor
       - \left\lfloor \frac{2 s}{m_i} \right\rfloor
   = \begin{cases}
       0 & \text{if $i$ is even or $i = 1$} \\
       1 & \text{if $i$ is odd and $i > 1$} \\
   \end{cases}
\end{equation*}
and
\begin{align*}
   \fc{1}{U}(n) &\geq |\rfact{l}{n}(U)|_1 - |\pref{n}(U)|_1 + 1 \\
   &= \sum_{i' = 2}^J (-1)^{(2i' - 1) + 1}
       + \sum_{i = 2J + 1}^\infty (-1)^{i + 1} \left(
           \left\lfloor \frac{n + l}{m_i} \right\rfloor -
           \left\lfloor \frac{l}{m_i} \right\rfloor -
           \left\lfloor \frac{n}{m_i} \right\rfloor \right) + 1\\
   &= J + 1 = \frac{n' + 1}{2}.
\end{align*}

Consider the third claim.
Let $n = m_j$.
Then
\begin{align*}
   |\rfact{l}{n}(U)|_1
   &= \sum_{i = 1}^\infty (-1)^{i + 1} \left(
       \left\lfloor \frac{m_j + l}{m_i} \right\rfloor -
       \left\lfloor \frac{l}{m_i} \right\rfloor \right)\\
   &= \sum_{i = 1}^j (-1)^{i + 1} \frac{m_j}{m_i}
       + \sum_{i = j + 1}^\infty (-1)^{i + 1} \left(
           \left\lfloor \frac{m_j + l}{m_i} \right\rfloor -
           \left\lfloor \frac{l}{m_i} \right\rfloor \right)
\end{align*}
and
\begin{equation*}
   \sum_{i = j + 1}^\infty (-1)^{i + 1} \left(
       \left\lfloor \frac{m_j + l}{m_i} \right\rfloor -
       \left\lfloor \frac{l}{m_i} \right\rfloor \right)
   \in \{0, (-1)^{j}\},
\end{equation*}
so $\fc{1}{U}(n) = 2$.
\qed\end{proof}

If $n_i = 2$ for all $i$, then the word $U$ is the period-doubling word $S$.
Thus Lemma \ref{lem:fcu} gives an alternative proof for Theorem \ref{thm:fcs}.

\begin{theorem} \label{thm:arbslow}
For every increasing unbounded function $f: \Ni \to \Ni$
there is a uniformly recurrent word $w \in \{0, 1\}^\omega$
such that $\fc{k}{w}(n) = O(f(n))$ but $\fc{k}{w}(n)$ is not bounded.
\end{theorem}

\begin{proof}
Follows from Lemmas \ref{lem:fckfc1}, \ref{lem:uunrec} and \ref{lem:fcu}.
\qed\end{proof}

\section{Conclusion} \label{sec:conclusion}

In this paper we have investigated some generalizations
of the results of Morse and Hedlund and those of Coven and Hedlund
for $k$-Abelian complexity.
We have pointed out many similarities but also many differences.
We have studied the $k$-Abelian complexity of the Thue-Morse word
and proved that there are uniformly recurrent words
with arbitrarily slowly growing $k$-Abelian complexities.

There are many open question and possible directions for future work.
One open problem related to Lemma \ref{lem:perlin} is to
determine the maximal $(k + 1)$-Abelian complexity
of a $k$-Abelian periodic word.
Another interesting topic would be
the $k$-Abelian complexities of morphic words.
For example, Theorem \ref{thm:arbslow} does not hold
if morphic words are considered instead of uniformly recurrent words,
because the $k$-Abelian complexity of a morphic word
is always a computable function.
But for a morphic (or pure morphic) word $w$, how slowly can
\begin{math}
   \fcu{k}{w}(n)
\end{math}
grow without being bounded?
Can it grow slower than logarithmically?
More generally, can the possible $k$-Abelian complexities
of some subclass of morphic words be classified?

%\bibliographystyle{plain}
%\bibliography{../bibtex/ref}

\end{document}